\documentclass[12pt]{article}
\usepackage[top=1in, bottom=1in, left=1in, right=1in]{geometry}
\usepackage{amsmath,amssymb}
\usepackage{graphicx,epstopdf}
\usepackage{hyperref}
\usepackage{amsthm}
\usepackage{color}
\usepackage{booktabs}
\usepackage{graphicx,epstopdf}
\usepackage{authblk}
\usepackage[usenames,dvipsnames]{pstricks}
\usepackage{epsfig}
\usepackage{pstricks-add}
\usepackage{pst-grad} 
\usepackage{pst-plot}

\newtheorem{thm}{Theorem}[section]

\newtheorem{con}{Conjecture}

\begin{document}
\title{\textbf{Total Colorings-A Survey }}
\author[1]{\textbf{J. Geetha}}
\author[2]{\textbf{Narayanan. N}}
\author[1]{\textbf{K. Somasundaram}}
\affil[1]{Department of Mathematics, Amrita School of Engineering-Coimbatore\\ Amrita Vishwa Vidyapeetham, India.}
\affil[2]{Department of Mathematics\\
Indian Institute of Technology Madras\\
Chennai, India.}
\date{}
\maketitle

\noindent \textbf{Abstract:} The smallest integer $k$ needed for the
assignment of colors to the elements so that the  coloring is proper
(vertices and edges) is called the
total chromatic number of a graph. Vizing~\cite{78} and Behzed~\cite{76,77}
conjectured that the total coloring can be done using at most
$\Delta(G)+2$ colors, where $\Delta(G)$ is the maximum degree of $G$.
 It is not
settled even for planar graphs. In this paper we give a survey on
total coloring of graphs.

\section{Introduction}
Let $G$ be  a simple graph with vertex set $V(G)$ and edge set $E(G)$. An element of $G$ is a vertex or an edge of $G$. The \textit{total coloring} of a graph $G$ is an assignment of colors to the vertices and edges such that no two
incident or adjacent elements receive the same color. The \textit{ total chromatic number} of $G$, denoted by
$\chi''(G)$, is the least number of colors required for a total coloring. Clearly, $\chi''(G)\geq \Delta(G)+1$, where
$\Delta(G)$ is the maximum degree of $G$.

Behzed~\cite{76,77} and Vizing~\cite{78} independently posed a
conjecture called Total Coloring Conjecture (TCC) which states that
for any simple graph $G$, the total chromatic number is either
$\Delta(G)+1$ or $\Delta(G)+2$.  In~\cite{86}, Molly and Reed gave a
probabilistic approach to prove that for sufficiently large
$\Delta(G)$, the total chromatic number is at most
$\Delta(G)+10^{26}$.

If $\chi''(G)=\Delta(G)+1$ then $G$ is known as a type-I  graph and if $\chi''(G)=\Delta(G)+2$  then $G$ is a type-II
graph. In this  paper, we present a comprehensive survey on total coloring. Yap ~\cite{79} gave a nice survey on
total colorings that covers the results till 1995. Therefore, our survey cover results from 1996 onwards. There are
four sections in this paper. In the second section, we focus on results on planar graphs. In ~\cite{80}, Borodin
gave a survey of results on total colorings of planar graphs up to 2009. There are several  improved  results on planar
graphs since then. Thus, we only present the results from 2010 onwards. For the earlier results we refer the readers to
the earlier two excellent surveys mentioned above. Third section of this paper consists of results on non-planar graphs.
In this paper, we also  prove TCC holds for unitary cayley graphs, mock threshold graphs and odd graphs. In the last section,
we survey the results on complexity aspects of total coloring.


\section{Planar graphs} In this section, we consider the results related to planar graphs (graphs that have a plane
embedding).  Many of the results in total coloring of planar graphs are based on the maximum degree and the girth
constraints.

One of the most yielding techniques on planar graphs is the Discharging Method.  While one can say that the discharging method have been used in graph theory  for more than 100 years, it came to prominence when it was used to prove the four color theorem by Appel and Hacken. Since then, the method has  been applied to many types
of problems (including graph embeddings and decompositions, spread of infections in networks, geometric problems,
etc.). It is especially useful for dealing with planar graphs. An excellent guide to  the method of  discharging is given by Cranston and West ~\cite{98}.

A rough sketch of  using the discharging method  is as follows~\cite{99}:

\textbf{Charging phase:}

1. Assign initial charges to certain elements of a graph (vertices, edges, faces, etc.,).

2. Compute the total charge assigned to the whole graph (for
planar graphs typically using Euler's formula).

\textbf{Discharging phase:}

3. Redistribute charge in the graph according to a set of discharging rules.

4. Compute the total charge again (using the specific properties of the graph),
and derive a conclusion.

A configuration in a graph $G$ can be any structure in $G$ (often a specified sort of subgraph). A configuration is reducible for a graph property $Q$ if it cannot occur in a minimal graph not having property $Q$. The method of
discharging is used to show that a set of reducible configurations is unavoidable in the class of graphs being discussed which  establishes that the property $Q$ cannot have a counterexample in the class. Let $d_G(v)$ or simply $d(v)$ denote the degree (number of neighbors) of vertex $v$ in $G$, and let $d(G)$  denote the average of the vertex degrees in $G$. Degree charging is the assignment to each vertex $v$ of an initial
charge equal to $d(v)$.   In the following, we present the  total coloring results on planar graphs.  \\

  The total coloring conjecture  was verified for planar graphs with $\Delta(G)\leq 5$ ~\cite{kos96}. Sanders and Zhao~\cite{95} showed that every planar graph with $\Delta(G)\leq 7$ is 9-total colorable. Yap ~\cite{79} verified TCC for planar graphs with $\Delta(G)= 8$. In ~\cite{kss08}, Kowalik et al. proved that any planar graph with $\Delta(G)\geq9$ is type-I.    Shen et al.~\cite{91} posed a conjecture on the total coloring of planar graph which states that ``Planar graphs with $4\leq \Delta(G)\leq 8$ are $\Delta(G)+1$-total colorable". Wang et al.~\cite{12}  proved that  for a planar graph $G$ with maximum degree $\Delta(G)$ and girth $g$ such that $G$ has no cycles of length from $g+1$ to $t, \ t>g$, the total chromatic number is $\Delta(G)+1$ provided $(\Delta(G), g, t)\in \{(5,4,6), (4,4,17)\}$ or $\Delta(G)=3$
and $(g,t)\in\{(5,13), (6,11), (7,11), (8,10), (9,10)\}$, where each vertex is incident with at most one $g$-cycle. For more details we  refer the excellent survey  by Borodin ~\cite{80}.

\begin{center}
    \scalebox{1} 
    {
        \begin{pspicture}(0,-2.5189064)(5.8984375,2.5189064)
            \pstriangle[linewidth=0.04,dimen=outer](2.6135938,0.72046876)(1.64,1.24)
            \psdots[dotsize=0.2](2.6335938,1.9604688)
            \psdots[dotsize=0.2](3.4135938,-0.33953115)
            \usefont{T1}{ptm}{m}{n}
            \rput(2.6264062,2.3304687){$u_1$}
            \usefont{T1}{ptm}{m}{n}
            \rput(4.0264063,0.85046875){$u_3$}
            \usefont{T1}{ptm}{m}{n}
            \rput(1.3464062,0.6904687){$u_2$}
            \usefont{T1}{ptm}{m}{n}
            \rput(4.0264063,-0.32953125){$u_4$}
            \usefont{T1}{ptm}{m}{n}
            \rput(4.0064063,-1.8095311){$u_5$}
            \psline[linewidth=0.04cm](3.4140625,0.73890626)(3.3940625,-1.6810937)
            \usefont{T1}{ptm}{m}{n}
            \rput(2.9232812,-2.2910938){Fig. 1.  Reducible configuration from ~\cite{91}}
        \end{pspicture}
    }
\end{center}

  First we start with results involving planar graphs with maximum degree at
least six.     As far as we know  the first work on the  total coloring with
$\Delta(G)=6$ was given by  Wang et al.~\cite{92}. They verified TCC for planar
graphs without 4-cycles.  Shen et
al. improved the result by showing that the planar graph without 4-cycles is
type-I~\cite{91} with the reducible configuration shown in
Fig.1.  Sun et al.~\cite{60} proved that every planar graph $G$ with maximum
degree 6 is totally 8-colorable if no two triangles in $G$ share a common edge
(which implies that every vertex $v$ in $G$ is incident with at most $\lfloor
\frac{d(v)}{2}\rfloor$ triangles. In other words,  every vertex is missing
either a 3-cycle or a 4-cycle). They also proved that a planar graph without
adjacent triangles and without cycles of length $k$,
$k\geq 5$ is type-I. Nicolas Roussel~\cite{62} strengthened the result of ~\cite{60} by
showing that  a planar graph $G$ is total 8-colorable if every vertex of $G$ is
missing some $k_v$-cycle for $k_v \in \{3,4,5,6,7,8\}$. In 2017,   Zhu and
Xu~\cite{51} improved the result of Sun et al.  ~\cite{60} to show that TCC holds for
planar graphs $G$ with $\Delta(G)=6,$ provided $G$ does not contain any subgraph
isomorphic to a 4-fan.

Further improvements on the total coloring of planar graph with $\Delta\geq 6$
as follows. Zhang and Wang~\cite{65} showed that every planar graph with
$\Delta\geq 6$ and without adjacent short cycles (a cycle of length at most 4)
is $\Delta+1$-total-colorable.  Hou et al.~\cite{8} proved that a planar graph
$G$ is type-I if $\Delta(G)\geq 5$ and $G$ contains neither 4-cycles nor
6-cycles or $\Delta(G)\geq 6$ and $G$ contains neither 5-cycles nor 6-cycles. A chordal $k$-cycle is a $k$-cycle with at least
one chord.
In~\cite{6}, Wu et al. improved the result of ~\cite{8} by showing that a planar graph with
maximum degree $\Delta(G)$ is $M$-totally colorable if it contains neither chordal
5-cycle nor chordal 6-cycle, where $M=\text{max}\{7,\Delta(G)+1\}$.
 Dong et al.~\cite{48}
proved that  a planar graph $G$ where no 6-cycles has a chord satisfies TCC
provided $\Delta(G)\geq 6$.  Li~\cite{50} verified TCC  for planar graphs  with
maximum degree six, if for each vertex $v$, there is an integer $k_v \in
\{3,4,5,6\}$ such that $G$ has no $k_v$-cycle  containing $v$.

A notion closely related to total coloring of graphs is  list total coloring of graphs. Suppose that a set $L(x)$ of
colors, called a list of $x$, is assigned to each element $x\in V(G)\cup E(G)$. A total coloring $\phi$ is called a list
total coloring of $G$ or $L$-coloring, if $\phi(G)\in L(x)$ for each element $x\in V(G)\cup E(G)$. If $|L(x)|=k$ for
every $x\in V(G)\cup E(G)$, then a total $L$- coloring is called a list total $k$ coloring and we say that $G$ is
totally $k$-choosable and the minimum integer $k$ for which $G$ is total $k$-choosable is the total choosability of $G$.
Liu  et al.~\cite{49} proved that a planar graph $G$ is total $\Delta(G)+2$-choosable ($\Delta(G)+2$-total colorable)
whenever (1) $\Delta(G)\geq 7$ and $G$ has no adjacent triangles or (2) $\Delta(G)\geq 6$ and $G$ has no intersecting
triangles or (3) $\Delta(G)\geq 5$ and $G$ has no adjacent triangles and $G$ has no $k$-cycles for some integer $k\in
\{5,6\}$.

We now list  the results on the total coloring of planar graph with maximum degree at least 7. The first work on this direction
is due to Sanders and Zhao~\cite{95}.
\begin{center}
    \scalebox{1} 
    {
        \begin{pspicture}(0,-3.9289062)(12.44,3.8889062)
            \pstriangle[linewidth=0.04,dimen=outer](0.69804156,2.5747952)(1.14,1.16)
            \psline[linewidth=0.04cm](1.2280415,2.5947952)(1.2280415,0.81479514)
            \psdots[dotsize=0.2](0.6880415,3.714795)
            \psdots[dotsize=0.2](1.2480415,1.6947951)
            \usefont{T1}{ptm}{m}{n}
            \rput(0.453125,0.7189062){a}
            \rput{-178.91699}(9.035453,6.4149876){\pstriangle[linewidth=0.04,dimen=outer](4.548042,2.5547953)(1.68,1.22)}
            \pstriangle[linewidth=0.04,dimen=outer](4.5480413,1.3747951)(1.68,1.22)
            \psdots[dotsize=0.2](5.3680415,3.774795)
            \psdots[dotsize=0.2](5.3480415,1.3947952)
            \psline[linewidth=0.04cm](5.3480415,1.4147952)(5.6680417,0.93479514)
            \usefont{T1}{ptm}{m}{n}
            \rput(4.644219,0.8389062){b}
            \psdiamond[linewidth=0.04,dimen=outer](8.278042,2.714795)(0.73,0.92)
            \psline[linewidth=0.04cm](7.5880413,2.714795)(8.948042,2.714795)
            \psdots[dotsize=0.2](8.268042,3.6347952)
            \psdots[dotsize=0.2](8.268042,1.7947952)
            \psline[linewidth=0.04cm](8.288041,1.7747952)(8.288041,0.7747951)
            \usefont{T1}{ptm}{m}{n}
            \rput(8.247969,0.4989063){c}
            \psdiamond[linewidth=0.04,dimen=outer](1.79,-1.0510937)(0.55,1.14)
            \psline[linewidth=0.04cm](0.7,-2.1310937)(2.92,-2.1310937)
            \psline[linewidth=0.04](2.32,-1.0510937)(3.48,-1.0510937)(2.92,-2.1310937)
            \psline[linewidth=0.04](1.24,-1.0510937)(0.02,-1.0510937)(0.7,-2.1310937)
            \psdots[dotsize=0.2](2.32,-1.0710938)
            \psdots[dotsize=0.2](1.3,-1.0510937)
            \psdots[dotsize=0.2](2.92,-2.1110938)
            \psdots[dotsize=0.2](0.74,-2.1110938)
            \pspolygon[linewidth=0.04](6.5,-2.1310937)(7.12,-1.3310938)(7.1,-0.57109374)(6.5,-1.3310938)(5.88,-0.5310938)(5.9,-1.3310938)
            \psline[linewidth=0.04cm](6.48,-1.3110938)(6.5,-2.0910938)
            \psline[linewidth=0.04](7.12,-1.3110938)(7.88,-1.2910937)(7.14,-2.1310937)(5.9,-2.1110938)(5.22,-1.3310938)(5.92,-1.3310938)
            \psdots[dotsize=0.2](6.48,-1.2910937)
            \psdots[dotsize=0.2](7.14,-1.3110938)
            \psdots[dotsize=0.2](7.14,-2.1310937)
            \psdots[dotsize=0.2](5.92,-1.3310938)
            \psdots[dotsize=0.2](5.92,-2.1310937)
            \pspolygon[linewidth=0.04](10.82,-2.2910938)(11.44,-1.4910938)(11.42,-0.73109376)(10.82,-1.4910938)(10.2,-0.69109374)(10.22,-1.4910938)
            \psline[linewidth=0.04cm](10.8,-1.4710938)(10.82,-2.2510939)
            \psline[linewidth=0.04](11.44,-1.4710938)(12.2,-1.4510938)(11.46,-2.2910938)(10.22,-2.2710938)(9.54,-1.4910938)(10.24,-1.4910938)
            \psline[linewidth=0.04](10.8,-2.2910938)(10.32,-3.1510937)(9.66,-3.1510937)(10.22,-2.2710938)
            \psdots[dotsize=0.2](10.34,-3.1710937)
            \psdots[dotsize=0.2](10.2,-2.3310938)
            \psdots[dotsize=0.2](10.2,-1.4910938)
            \psdots[dotsize=0.2](10.84,-1.4510938)
            \psdots[dotsize=0.2](11.46,-1.5110937)
            \psdots[dotsize=0.2](11.48,-2.2710938)
            \pspolygon[linewidth=0.04](11.46,1.3089062)(12.32,2.5089064)(12.32,3.6489062)(11.5,2.5089064)(10.68,3.7289062)(10.68,2.4689062)(10.7,2.4489062)
            \psline[linewidth=0.04cm](11.5,2.5289063)(11.48,1.3289063)
            \psdots[dotsize=0.2](11.5,2.5089064)
            \psdots[dotsize=0.2](10.7,2.4689062)
            \psdots[dotsize=0.2](12.32,2.4689062)
            \usefont{T1}{ptm}{m}{n}
            \rput(11.464531,0.93890625){d}
            \usefont{T1}{ptm}{m}{n}
            \rput(1.668125,-2.5210938){e}
            \usefont{T1}{ptm}{m}{n}
            \rput(6.5314064,-2.5610938){f}
            \usefont{T1}{ptm}{m}{n}
            \rput(11.281406,-3.0010939){g}
            \usefont{T1}{ptm}{m}{n}
            \rput(5.4979687,-3.7010937){Fig. 2.  Reducible configurations from ~\cite{5}}
        \end{pspicture}
        }

    \end{center}

    Shen and Wang~\cite{5} showed that the planar graphs with maximum degree 7 and without 5-cycles are 8-totally colorable. Fig.2. shows the reducible configuration used in ~\cite{5}.
    Chang et al. ~\cite{63} proved that a planar graph $G$ with maximum
    degree 7, with the additional property that for every vertex $v$, there is an integer $k_v \in
    \{3,4,5,6\}$ so that $v$ is not incident with any $k_v$-cycle, is type-I.  Wang and  Wu~\cite{46} proved that a planar
    graph of maximum degree $\Delta(G)\geq 7$ is $\Delta(G)+1$-totally colorable if no 3-cycle has a common vertex with a
    4-cycle or no 3-cycle is adjacent to a cycle of length less than 6.
    In~\cite{13}, Wang et al.   proved that for any planar graph
    with maximum degree $\Delta(G)\geq 7$ and without intersecting 3-cycles (two cycles of length 3 are not incident with a
    common vertex), and without intersecting 5-cycles,  the total chromatic number is $\Delta(G)+1$~\cite{14}.    Wu  et al.~\cite{1} proved that for a planar graph
    with maximum degree at least 7 and without adjacent 4-cycles the total chromatic number is $\Delta(G)+1$.

    The total chromatic number of a planar graph with $\Delta(G)\geq 7$ and without chordal 6-cycles~\cite{15} and without
    chordal 7-cycles~\cite{68} is $\Delta(G)+1$.

    We now turn to the results on maximum degree at least 8. There are many works on planar graphs with maximum degree at
    least 8. Yap~\cite{79} verified the TCC for planar graphs with $\Delta(G)\geq 8$.

    Roussel and  Zhu~\cite{7} proved that for a planar graph $G$ with maximum degree 8, $\chi''(G)=9$, if there is no
    $k_x$-cycle which contains $x$, where $x$ is a vertex in $G$ and $k_x\in \{3,4,5,6,7,8\}$.  This is an improvement
    over~\cite{103}.  The reducible configurations used by Roussel and  Zhu  are given in Fig.3. Further, Wang et al.~\cite{47}, strengthened the result and proved that for a planar graph with
    $\Delta(G)\geq 8$, if for every vertex $v \in V$, there exists two integers, $i_v, j_v \in \{3,4,5,6,7,8\}$ such that
    $v$ is not incident with intersecting $i_v$-cycles and $j_v$-cycles, then the total chromatic number of $G$ is
    $\Delta(G)+1$. Wang et al.~\cite{9} showed that  the total chromatic number of planar graphs with $\Delta(G)\geq 8$ is
    $\Delta(G)+1$, if for every vertex $v\in V(G)$, there exists two integers $i_v, j_v \in \{3,4,5,6,7\}$ such that $v$ is
    not incident with adjacent $i_v-$cycles and $j_v-$ cycles.
    \begin{center}
        \scalebox{1} 
        {
            \begin{pspicture}(0,-1.7489063)(4.67375,1.7489063)
                \psdiamond[linewidth=0.04,dimen=outer](1.85375,0.45890626)(1.54,1.19)
                \psline[linewidth=0.04cm](1.85375,1.6089063)(1.85375,-0.69109374)
                \psline[linewidth=0.04cm](3.37375,0.46890625)(4.59375,0.44890624)
                \psdots[dotsize=0.2](1.85375,1.6289062)
                \psdots[dotsize=0.2](0.33375,0.46890625)
                \psdots[dotsize=0.2](1.85375,-0.69109374)
                \psdots[dotsize=0.2](4.55375,0.44890624)
                \psdots[dotsize=0.2](3.35375,0.46890625)
                \usefont{T1}{ptm}{m}{n}
                \rput(1.7698437,-1.5210937){Fig. 3.  Reducible configuration from                 ~\cite{7}}
            \end{pspicture}
        }
    \end{center}

    Xu and  Wu~\cite{66} proved that if $G$ is a planar graph with maximum degree at least 8 and every 7-cycle of $G$
    contains at most two chords, then $G$ has a $(\Delta(G)+1)$-total-coloring.  Wang et al.~\cite{69}, considered planar
    graphs $G$ with maximum degree $\Delta(G)\geq 8,$ and showed that if $G$ contains no adjacent $i,j$-cycles with two
    chords for some $i,j\in \{5,6,7\}$, then $G$ is total $(\Delta(G)+1)$-colorable. Jian Chang et al.~\cite{11} proved that
    planar graphs with maximum degree at least 8 and without 5-cycles with two chords  are $\Delta(G)+1$ total colorable.

    In~\cite{70}, Cai et al. proved that the planar graphs with maximum degree 8 and without intersecting chordal 4-cycles are
    9-totally colorable.  Xu et al.~\cite{74} proved that if $G$ is a planar graph with maximum degree at least 8 and every
    6-cycle of $G$ contains at most one chord or any chordal 6-cycles are not adjacent, then $G$ has a $\Delta(G)+1$-total
    coloring.  In 2014, Wang et al.~\cite{16} showed that a planar graph with $\Delta(G)\geq 8$ and without adjacent cycles
    of size $i$ and $j$, for some $3\leq i\leq j \leq 5$, is $(\Delta(G)+1)$-total colorable.

    Later, Wang et al.~\cite{10} proved that a planar graph with $\Delta(G)\geq 8$ and if $v\in V(G)$ is not incident with
    chordal 6-cycle, or chordal 7-cycle or 2-chordal  5-cycle, then it is type-I. This is a generalisation of ~\cite{16}.

    In~\cite{17}, Chang et al. proved that a planar graph with maximum degree 8 is 9-total colorable if for every vertex
    $v$, $v$ is incident with at most $d(v)-2\lfloor\frac{d(v)}{5} \rfloor$ triangles. This is a generalisation of ~\cite{70} and ~\cite{11}.


    There are some other classes  of graphs which are similar to planar graphs. We discuss the total coloring of some of
    them.

    A graph is called 1-planar, if it has at most one crossing edge. The
    following are some results on total coloring of
    1-planar graphs.   Zhang et al.~\cite{45} proved that each 1-planar graph with $\Delta(G)\geq 16 $ is
    $(\Delta(G)+2)$-total colorable and $(\Delta(G)+1)$-total colorable if $\Delta(G)\geq 21$. J$\acute{\text u}$lius
    Czap~\cite{2} studied the 1-planar graphs and gave some upper bound for the total chromatic number. He showed that a
    1-planar graph with $\Delta(G) \geq 10$ satisfies TCC if $\chi(G)\leq 4$. He also proved that if $G$ is a 1-planar graph
    without adjacent triangles and with $\Delta(G) \geq 8$, then $\chi''(G)\leq \Delta(G)+3$ and if $\chi(G)\leq 4$, then
    $\chi''(G)\leq \Delta(G)+2$. Xin Zhang et al.~\cite{4} showed that for a 1-planar graph $G$, if  $\Delta(G)\leq r$ and
    $r\geq 13$, where $r$ is an integer, then $\chi''(G)\leq r+2$.

    An outerplanar graph is a planar graph that has a plane embedding such
    that all vertices lie on the boundary of the outer face.  In~\cite{3},
    Y. Wang and W.  Wang characterized the adjacent vertex distinguishing
    total chromatic number of outerplanar graphs. They proved that, if $G$
    is an outerplane graph with $\Delta(G)\geq 4$, then $\chi''_a(G)\leq
    \Delta(G)+2$ and so the TCC is satisfied. They also proved that, if
    $G$ is an outerplane graph with $\Delta(G)\geq 4$ and without adjacent
    vertices of maximum degree, then $\chi''_a(G)= \Delta(G)+1$ and hence
    $G$ is type-I graph.

    A graph is pseudo-outerplanar if each of its blocks has an embedding in the plane so that the vertices lie on a fixed
    circle and the edges lie inside the disk of this circle with each of them crossing at most one another.  In~\cite{67},
    Xin Zhang and Guizhen Liu verified TCC for pseudo-outerplanar graphs and proved that the total chromatic number of every
    outerplanar graph with $\Delta(G)\geq 5$ is $\Delta(G)+1$.  Xin Zhang~\cite{64} proved that every pseudo-outerplanar
    graph with $\Delta(G)\geq 5$ is totally $\Delta(G)+1$-choosable and
    hence the total chromatic number also has this as upper bound.

    Similar to planar graphs there is another type of graphs called toroidal graphs which  are embedding of graphs on torus
    such that there is no crossing edges.  Tao Wang~\cite{72} proved that if $G$ is a 1-toroidal graph with maximum degree
    $\Delta(G)\geq 11$ and without adjacent triangles, then $G$ has a total coloring with at most $\Delta(G)+2$ colors.


    \section{Non-planar graphs}

    \subsection{Circulant Graph}

    For a sequence of positive integers $1\leq d_1<d_2<...<d_l\leq  \lfloor \frac{n}{2} \rfloor$, the circulant graph
    $G=C_n(d_1,d_2,...,d_l)$ has vertex set $V=Z_n=\{0,1,2,...,n-1\}$, two vertices $x$ and $y$ being adjacent iff $x=(y\pm
    d_i) \mod n$ for some $i, 1\leq i\leq l.$

    Riadh Khennoufa and  Olivier Togni~\cite{61} studied total colorings of circulant graphs and proved that every 4-regular
    circulant  graphs $G=C_{5p}(1,k)$ and $C_{6p}(1,k)$  are type-I graphs for any positive integer $p$ and $k<\frac{5p}{2}$
    with $k\equiv 2 \mod 5$ or $k \equiv 3 \mod 5$ and $p\geq 3$ and $k< 3p$ with $k\equiv 1 \mod 3$ or $k\equiv 2 \mod 3$
    respectively.

    A graph is a \textit{power of cycle}, denoted $C^k_n$, $n$ and $k$ are integers, $1\leq k<\lfloor\frac{n}{2}\rfloor$, if
    $V(C^k_n)=\{v_0,v_1,...,v_{n-1}\}$ and $E(C^k_n)=E^1 \cup E^2 \cup ...\cup E^k$, where
    $E^i=\{e_0^i,e_1^i,...,e^i_{n-1}\}$ and $e_j^i=(v_j,v_{(j+i) \mod \ n })$ and $0 \leq j \leq n-1$, and $1\leq i\leq k$.

    Campos and de Mello~\cite{39} proved that $C_n^2, n\neq 7$, is type-I and $C_7^2$ is type-II. They ~\cite{18} 
    verified the TCC for powers of cycles $C_n^k, n$ even and
    $2<k<\frac{n}{2}$ and also showed that one  can obtain a
    $\Delta(G)+2$-total coloring for these graphs in polynomial time. They also proved that $C_n^k$ with $n\cong 0 \mod (\Delta(C_n^k)+1)$ are
    type-I and they proposed the following conjecture.  \begin{con} Let $G=C_n^k$, with $2\leq k<\lfloor \frac{n}{2}
        \rfloor$. Then,

        $\chi''(G)=\begin{cases}\Delta(G)+2, & \text{ if }k>\frac{n}{3}-1  \  \text{and } n \ \text{is \ odd}\\ \Delta(G)+1  &
        \text{ otherwise}.  \end{cases}$

    \end{con}

    Geetha et al.~\cite{96} proved this conjecture for certain values of $n$ and $k$. They also verified  TCC for the
    complement of powers of cycles $\overline{C^k_n}$. In particular, they proved that  $\overline{C_n^2}$ is type-II for $n\leq 8$.

    Cayley graphs are those whose vertices are the elements of groups and adjacency relations are defined by subsets of the
    groups. Cayley graphs contain long paths and have many other nice combinatorial properties. They have  been used to
    construct other combinatorial structures. Also, for the constructions of various communication networks, and difference
    sets in design theory. Cayley graphs have been used to analyze algorithms for computing with groups.

    Let $\Gamma$ be a multiplicative group with identity 1. For $S\subseteq \Gamma, 1\notin S \text{ and }
    S^{-1}=\{s^{-1}:s\in S\}=S$ the \textit{Cayley Graph} $X=Cay(\Gamma, S)$ is the undirected graph having vertex set
    $V(X)=\Gamma$ and edge set $E(X)=\{(a,b): ab^{-1}\in S\}$.

    For a positive integer $n>1$ the \textit{unitary Cayley graph} $X_n=Cay(Z_n, U_n)$ is defined by the additive group of
    the ring $Z_n$ of integer modulo $n$ and the multiplicative group $U_n$ of its units. If we represent the elements of
    $Z_n$ by the integers 0,1,...$n-1$, then it is well known that \begin{center} $U_n=\{a\in Z_n: gcd(a,n)=1\}$.
    \end{center} So $X_n$ has vertex set $V(X_n)=X_n=\{0,1,2,...,n-1\}$ and edge set \begin{center} $E(X_n)=\{(a,b): a,b \in
    Z_n, gcd(a-b,n)=1\}$. \end{center} Boggess et al. ~\cite{bhjk08} studied the structure of unitary cayley graphs. They
    have also discussed chromatic number, vertex and edge connectivity, planarity and crossing number. Klotz and Sander
    ~\cite{ws07} have determined the clique number, the independence number and the diameter. They have given a necessary
    and sufficient condition for the perfectness  of $X_n$.

    The graph $X_n$ is regular of degree $U_n=\varphi(n)$ denotes the Euler function. Let the prime factorization of $n$ be
    $p_1^{\alpha_1} p_2^{\alpha_2}...p_t^{\alpha_t}$ where $p_1<p_2<...<p_t$.  If $n=p$ is a prime number, then $X_n=K_p$ is
    the complete graph on $p$ vertices. If $n=p^\alpha$ is a prime power then $X_n$ is a complete $p$-partite graph. In the
    following theorem we prove TCC holds for unitary Cayley graphs.

    \begin{thm} A unitary Cayley graph $X_n$ is $(\Delta(X_n)+2)$ - total
    colorable.  \end{thm}

    \begin{proof} We know that a unitary Cayley graph can be obtained from
        a balanced $r$ partite     graph by deleting some edges. Suppose
        $n=p$ is a prime number, then $X_n$ is the complete graph on $p$
        vertices.  Also, if $n=p^\alpha$, a prime power, then $X_n$ is a
        complete $p$-partite graph and TCC holds for these two graphs
        ~\cite{79}.

        When $n=2k, k \in N$, then the unitary Cayley graph is a bipartite
        graph and any bipartite graph is total colorable.

        Suppose $n \not \equiv 0 \mod 2 $. As $p_1$ is the smallest prime,
        $kp_1,kp_1+1,...,(k+1)p_1-1,$ where $k=0,1,2,...,\frac{n}{p_1}-1$
        induces $\frac{n}{p_1}$  vertex disjoint cliques each of order $p_1$.
        Since $p_1$ is odd, we can color all the elements of these
        $\frac{n}{p_1}$ cliques using $p_1$ colors ~\cite{79}. Now remove the
        edges of these cliques. The remaining graph is a
        $\varphi(n)-p_1+1$-regular graph where the vertices are already
        coloured. We  color the  edges of this resultant graph with
        $\varphi(n)-p_1+2$ colors. Thus we have used $\varphi(n)+2$ colors
    for the total coloring of $X_n$.  \end{proof}

    In the following section, we look at graph products and papers on
    total colouring of product graphs.

    \subsection{Product Graphs} Graph products were first defined by Sabidussi~\cite{93} and Vizing~\cite{94}. A lot of work
    has been done on various topics related to graph products, but on the other hand there are still many questions open. There
    are four standard graph products, namely, $cartesian \ product$ ($G \Box H$),  $direct \ product$ ($G \times H$),
    $strong \ product$ ($G\boxtimes H$) and  \textit{lexicographic  product} ($G \circ H$).  In~\cite{93}, these products
    have been widely discussed with significant applications. The vertex
    sets of these products are same: $V(G) \times V(H)$.
    The edge sets are $E(G \Box H)=\{((g,h),(g',h'))| \ g=g', \ hh'\in E(H), \ \text{or} \ gg' \in E(G), \ h=h'\}$, \\ $E(G
    \times H)= \{((g,h),(g',h'))| \ gg' \in E(G) \ \text{and} \ hh'\in E(H)\}$, \\ $E(G\boxtimes H)=E(G\Box H)\cup E(G\times
    H)$ and \\ $E(G \circ H)=\{((g,h),(g',h'))| \ g=g', \ hh'\in E(H), \ \text{or} \ gg' \in E(G) \}$. \\ The first three
    products are commutative and the lexicographic product is associative but  not commutative.\\

    The total coloring conjecture was verified for the cartesian product
    of two graphs. Seoud et al.~\cite{35, SMWW97}  determined the total
    chromatic number of the join of two paths, the cartesian product of
    two paths, the cartesian product of a path and a cycle, certain
    classes of the corona of two graphs and the theta graphs.   Kemnitz
    and Marangio ~\cite{36} classified the cartesian product of complete
    graphs $K_n \Box K_m$ as type-I if $n\geq m\geq 4, n\equiv 0 \mod 4$
    or $n>m \geq 4, n\equiv 2 \mod 4$, where $n$ and $m$ are even and as
    type-II if $n$ is even  and $m$ odd and $n>(m-1)^2$. They also
    obtained the total chromatic number of the cartesian product of two
    cycles $C_n \Box C_m$,   $K_n \Box H$ and $C_n \Box H$, where
    $H$ is a bipartite graph. Using the fact that if $G$ is a regular
    graph with the adjacent vertex distinguishing chromatic index
    $\chi'_a=\Delta(G)+1$ then $\chi''(G)=\Delta(G)+1$,  Baril et al. ~\cite{BHO12} proved that $K_n
    \Box K_m$ is type-I if $m\text{ and }n$ are even. Still there are cases of these
    products of complete graphs which is not classified as type-I or
    type-II.  Equitable total chromatic number of a graph $G$ is the
    smallest integer $k$ for which $G$ has a $k$-total coloring such that
    the number of vertices and edges colored with each color differs by at
    most one. Tong Chunling et al. ~\cite{22} improved the  results of
    Seoud et al. ~\cite{SMWW97} and  Kemnitz et al. ~\cite{36}  by showing
    that the cartesian product of two cycles $C_m$ and $C_n, m, n\geq 3$
    have equitable total 5-coloring. That is, $C_n \Box C_m$ is type-I.
    Zmasek and $\check{\text Z}$erovnik~\cite{19} proved that if TCC holds
    for graphs $G$ and $H$, then it holds for the cartesian product $G
    \Box H$. They also proved that if the factor with largest vertex
    degree is of type-I, then the product is also of type-I. \\

    There are only a few results proved on total colorings of the other
    three product graphs. Katja Prnaver and Bla$\check{\text{z}}$
    Zmazek~\cite{24} verified the conjecture for direct product of a path
    and  any graph $G$ with $\chi'(G)=\Delta(G).$ Geetha and
    Somasundaram~\cite{58} proved that  direct
    product of two even complete graphs are type-I and the direct product
    of two cycles $C_m$ and $C_n$ are type-I for certain values of  $m$ and $n$.
    They also proved that if $K_2\boxtimes H$ ($K_2 \circ H$) satisfies TCC,
    then $G\boxtimes H$ ($G \circ H$)  satisfy  TCC, where $G$ is  any
    bipartite graph.  Mohan et al.~\cite{59} proved that the corona product
    of two graphs $G$ and $H$ is always type-I, provided $G$ is  total
    colorable and $H$ is either a cycle, a complete graph or a bipartite graph.
    The \textit{deleted lexicographic product} of two graphs $G$ and $H$,
    denoted by $D_{lex}$ $(G,H)$, is a graph with the vertex set $V(G)
    \times V(H)$ and the edge set $\{((g, h), (g', h') ) :  (g, g') \in
    E(G)  $ and  $ h \neq h'$, or  $\ (h, h') \in E(H)$  and $ g=g' \}$.
    Similar to lexicographic product, $D_{lex}(G, H)$ and $D_{lex}(H, G)$
    are not necessarily isomorphic.  Recently, Vignesh et al.
    ~\cite{VGS18} proved that the if $G$ is a bipartite graph and $H$ is  any total colorable graph then $G\circ H$ is also total colorable. They further
    show that  for any class-I graph $G$ and any graph $H$ with at least 3
    vertices,  $D_{lex}(G, H)$ is total colorable. In particular, if $H$
    is class-I then $D_{lex}(G, H)$ is also type-I.

    We present now the results on Sierpi$\acute{\text n}$ski graphs.

    \subsection{Sierpi$\acute{\text n}$ski Graphs}

    The Sierpi$\acute{\text n}$ski graphs $S(n,K_k)$, $k,n\geq1$, $k,n \in \mathbb{N} $ is defined on the  vertex set $\{1,
    2, ...,k\}^n$, where $K_k$ is complete graphs on $k$ vertices. Two different vertices $u=(u_1, u_2,..., u_n)$ and
    $v=(v_1, v_2,...,v_n)$ are adjacent if and only if there exists a $h\in \{1, 2, ..., n\}$ such that\\ \indent a)
    $u_t=v_t$ for $t=1,2,...,h-1$;\\ \indent b) $u_h\neq v_h$; and\\ \indent c) $u_t=v_h  \text { and }  v_t=u_h$ for
    $t=h+1,..., n$.\\

    Sierpi$\acute{\text n}$ski gasket graphs $S_n$ were introduced by Scorer, Grundy and Smith~\cite{40}. The graph $S_n$ is
    obtained from the Sierpi$\acute{\text n}$ski graphs $S(n,3)$ by contracting every edge of $S(n,3)$ that lies in no
    triangle. Marko Jakovac and Sandi Klav$\check{\text z}$ar~\cite{90} generalized the graphs $S(n,3)$ to
    Sierpi$\acute{\text n}$ski graphs $S(n,k)$ for $k\geq 3$ and determined the total colorings of the Sierpi$\acute{\text
    n}$ski gasket graphs $S_n$. In particular they proved that for any $n\geq 2$ and any odd $k \geq 3$, $S(n, k)$ and $S(n,
    4)$ are type-I graphs. For the even values of $k \geq 6$, they believed that $S(n, k)$ is always type-II and hence they
    proposed a conjecture that $S(n, k)$ is type-II. After three years Andreas M. Hinz, Daniele Parisse ~\cite{20} disproved
    the conjecture based on the canonical total colorings. Also they prove that the  Hanoi graphs $H_p^n$  are type-I
    graphs.  Geetha and Somasundaram~\cite{57} considered the generalized Sierpi$\acute{\text n}$ski graphs $S(n, G)$ and
    proved that $S(n, G)$ is type-I for certain classes of $G$.

    We now turn our attention to Chordal graphs.

    \subsection{Chordal Graphs}

    Chordal graphs are graphs in which every induced cycle is a 3-cycle. They form a very important class  of graphs due to
    the fact that  they have good algorithmic  properties. The TCC is verified for several subfamilies
    of chordal graphs like interval graphs, split graphs and strongly chordal graphs. A graph $G$ is called split graph if
    its vertex set can be partitioned in to two subsets $U$ and $V$ such that $U$ induces a clique and $V$ is independent
    set in $G$.  A color diagram $\mathcal{C}=\{R_1, R_2,..., R_k\}$ of frame $d=(d_1,d_2,...,d_k)$ is an ordered set of color arrays, where color array $R_i=\{c_{i,1},c_{i,2},...,c_{i,d_i}\}$, of length $d_i$, consists of distinct colors for all $1\leq i\leq k$. In~\cite{27}, Chen et al. proved that the split graphs satisfies TCC. They also proved that if $G$ is a split graph with
    $\Delta(G)$ even, then $G$ is type-I.  They extensively used  the concept of color diagram to prove these
    results.
    Campos et al.~\cite{28} gave  conditions for the split-indifference graph $G$ to be type-II and constructed a
    $\Delta(G)+1$-total colorings for the remaining.  Hilton~\cite{97} proved the following (it is known as Hilton's
    condition):     Let $G$ be a simple graph with an even number of vertices. If $G$ has a universal vertex, then $G$ is
    type-II if and only if $\left| E(\overline{G}) \right| + \alpha(\overline{G}) < \frac{|V(G)|}{2}$, where $\alpha(\bar{G})$ is the
    cardinality of a maximum independent set of edges of $\overline{G}$.

     Three-clique graphs are generalization of the
    split-indifference graphs. Campos et al.~\cite{28} proposed a conjecture based on the Hilton's condition.

    \begin{con} A 3-clique graph is type-II if and only if it satisfies Hilton's condition.  \end{con}

        A graph is dually chordal if it is the clique graph of a chordal graph. The class of dually chordal graphs
        generalize known subclasses of chordal graphs such as doubly chordal graphs, strongly chordal graphs, interval
        graphs, and indifference graphs. Figueiredo et al.~\cite{41}  proved that TCC holds for dually chordal graphs. A
        pullback from $G$ to $G'$ is a function $f: V(G) \rightarrow V(G')$, such that: (i) $f$ is a homomorphism and
        (ii) $f$ is injective when restricted to the neighborhood of $x, \forall x\in V(G)$. Based on this pullback method,
        they proved that if $\Delta(G)$ is even, then $G$ is type-I. A family of sets satisfies the Helly property if
        any subfamily of pairwise intersecting sets has nonempty intersection. A graph is neighborhood-Helly when the
        set $\{N(v): v \in V(G) \}$ satisfies the Helly property. A characterization of dually chordal graphs says that
        $G$ is dually chordal if and only if $G$ is neighborhood-Helly and $G^2$ is chordal. It is proved that~\cite{41}
        TCC  holds for neighborhood-Helly graphs $G$ such that $G^2$ is perfect.  They also proposed the
        following  problem which is still open:

        Determine the largest graph class for which all its odd maximum degree graphs are
        class-I and for which all its even maximum degree graphs are type-I.\\

        A graph is weakly chordal if neither the graph nor the complement of the graph has an  induced  cycle  on  five or  more
        vertices. A simple graph $G$ on $[n]= \{1, 2, ... , n\}$ is  threshold, if  $G$ can be built sequentially from
        the empty graph by adding vertices one at a time, where each new vertex is either isolated (nonadjacent to all the
        previous) or dominant (connected to all the previous).  A graph $G$ is said to be mock threshold if there is a vertex
        ordering $v_1, . . . , v_n$ such that for every $i \ (1 \leq i \leq n)$ the degree of $v_i$ in $G : {v_1, . . . , v_i}$ is
        0, 1, $i-2$, or $i-1$. Mock threshold graphs are a simple generalization of threshold graphs that, like threshold
        graphs, are perfect graphs. Mock threshold graphs are perfect and indeed weakly chordal but not necessarily chordal
        ~\cite{bsz18}. Similarly, the complement of a mock threshold graph is also mock threshold.

        In the following, we prove the TCC for Mock Threshold graphs.

        \noindent \textbf{Note:} A total coloring of $K_n$ can be constructed as follows: (This total coloring is due to Hinz and Parisse ~\cite{20})

        When $n$ is even, we first construct an edge coloring of $K_n$ and extend it. We denote $[n]_0=\{0,1,2,...,n-1\} $. For $k\in [n]_0$,
        let $\tau_k$ be the transposition of $k$ and $n-1$ on $[n]_0$. For even $n$, $c_n(i,j)=(\tau_i(j)+\tau_j(i)+2)\mod
        (n+1)$, for $i,j \in [n]_0, i\neq j$, defines a $(n+1)$-edge coloring. In this coloring assignment line $k\in [n]_0$
        will have the missing colors $k$ and $(k+1)\mod n$. We color $c_n(i)=i$ for all $i\in [n]_0$.\vspace{0.3cm}

        When  $n$ is odd,  we use the same coloring of $K_{n-1}$. In the coloring assignment of $K_{n-1}$, still the color
        $(k+1) \mod n$ is missing in line $k\in [n]_0$. We use these colors to the edges incident with $n^{th}$ vertex and
        color $n$ to the $n^{th}$ vertex.

        \begin{thm} Total coloring conjecture holds for any   Mock threshold graph $G$.  \end{thm} \begin{proof} Consider the
            Mock threshold graph $G$ with vertex ordering  $v_1, v_2, ..., v_i,..., v_n$.

            We prove this theorem using mathematical induction on the induced subgraph $G[v_1,v_2,...,v_k]$.

            For $k\leq 4$, the maximum degree of all the induced subgraphs is less than or equal to 3. We know that a graph with
            maximum degree less than or equal to 3 satisfies TCC ~\cite{kos96}.\vspace{0.3cm}

            Let us assume that $G[v_1,v_2,...,v_k], k\geq 5$  satisfies TCC.\vspace{0.3cm}

            \noindent \textbf{Claim:} The graph $G[v_1,v_2,...,v_k,v_{k+1}]$ satisfies TCC.\vspace{0.3cm}

            The degree of the vertex $v_{k+1}$ in $G[v_1,v_2,...,v_{k+1}]$  can be $0, 1,k-1$ or $k$.

            \noindent Case-1: Suppose  $d(v_{k+1})=0$.

            In this case the vertex is $v_{k+1}$ is an isolated vertex.  By the induction assumption, $G[v_1,v_2,...,v_k,v_{k+1}]$
            satisfies TCC.

            \noindent Case-2: Suppose $d(v_{k+1})=1$.

            In this case, the vertex $v_{k+1}$ is adjacent to a vertex, say $v_i$, in $G[v_1, v_2, ..., v_k]$. Since
            $G[v_1,v_2,...,v_k]$ is total colorable graph with at most $\Delta(G[v_1,v_2,...,v_k])+2$ colors, at each vertex there
            will be at least one missing color.  We assign this missing color to the edge $(v_i,v_{k+1})$, and  for the vertex
            $v_{k+1}$, we assign a color of a vertex which is not adjacent to $v_{k+1}$ and not the color of $v_i$.  Therefore,
            $G[v_1, v_2, ..., v_{k+1}]$ satisfies TCC.

            \noindent Case-3: Suppose  $d(v_{k+1})=k-1$.

            Let us  assume that the vertex $v_{k+1}$ is not adjacent with $v_i$ and  also assume that \\ $\Delta(G[v_1, v_2, ...,
            v_{k+1}])=k-1$. We consider following two cases:

            \noindent Subcase-1:  $k$ is even.

            Since $k$ is even, $k+1$ is odd. Construct a complete graph induced by the vertices
            $v_1,v_2,...v_{i-1},v_{i+1},...,v_{k+1}$.  Now, color this complete graph using  colors in the set $\{0,1,..., n+1\}$
            as given in  the note.  In this coloring assignment there will be one missing color at each of the vertices and they
            are distinct. Now, color the edges  $(v_i,v_j), i\neq j, j=1,2,...,v_{k+1}$,  with the missing colors. Assign the
            color $n-1$ to the vertex $v_i$. To get a total coloring of $G[v_1,v_2,...,v_{k},v_{k+1}]$, we remove these edges and
            there is no change in the maximum degree.

            \noindent Subcase-2:  $k$ is odd.

            In this case $k+1$ is even, say $2p$. It is known that  a graph of order $2p$ with maximum degree $2p-2$ satisfies TCC
            (see ~\cite{hil90, chen92}).

            \noindent Case-4: Suppose $d(v_{k+1})=k$.

            The  maximum degree of  $G[v_1,v_2,...,v_k,v_{k+1}]$ is   $k$. Construct a complete graph  on the vertex set $\{v_1,
            v_2, ..., v_k,v_{k+1}\}$. We know that the complete graph satisfies TCC. After removing the added edges we get a
            total coloring of $G[v_1,v_2,...,v_k,v_{k+1}]$

            Hence, in all the cases, the mock threshold graph satisfies TCC.

    \end{proof}


    \subsection{Multipartite Graphs}   
    
    Graph amalgamation~\cite{100} is one of the powerful techniques for various
    graph problems. A graph $H$ is an amalgamation of a graph $G$ if there exists a function $\phi$ called an amalgamation
    function from $V(G)$ onto $V(H)$ and a bijection $\phi': E(G)\rightarrow E(H)$ such that $e$ joining $u$ and $v$ is in
    $E(G)$ if and only if $\phi'(e)$ joining $\phi(u)$ and $\phi(v)$ is in $E(H)$.  Total coloring conjecture was verified
    for some classes of multipartite graphs using the amalgamation technique.

    Dong and Yap~\cite{83} proved that the complete $p$-prartite graph $K=K(r_1,r_2,..., r_p)$ is of type-I if $r_2\leq r_3-2$
    and $|V(K)|=2n$, $r_1\leq r_2\leq ...\leq r_p$. Deficiency of a graph $G$ to be def$(G)= \sum_{v\in V(G)} (\Delta(G) - d(v))$. Dalal and  Rodger~\cite{82}  proved that $K = K(r_1, . . . , r_5)$ is
    type-II if and only if $|V(K)|=0$(mod 2) and def($K$) is less than the number of parts in $K$ of odd size. Dalal et
    al.~\cite{56} proved that the complete $p$-partite graph $K=K(r_1,r_2,..., r_p)$ is type-I if and only if $K\neq
    K_{r,r}$ and if $K$ has an even number of vertices then def($K$) is at least the number of parts of odd size. Using  graph amalgamations technique they showed that all complete multipartite graphs of the form $K(r,r,...,r+1)$ are type-I .

    Chen et al.~\cite{89} proved that an $(n-2)$-regular equi-bipartite graph $K_{n,n}-E(J)$ is type-I if and only if $J$
    contains a 4-cycle. Campos and de Mello~\cite{31} determined the total chromatic number of some bipartite graphs like
    grids, near-ladders and $k$-dimensional cubes.\\


    \subsection{Cubic Graphs} 
    
    In~\cite{23}  Dantas et al. proved that for each integer $k\geq 2,$ there exists an integer
    $N(k)$ such that, for any $n\geq N(k)$, the generalized Petersen graph $G(n,k)$ has total chromatic number 4.

    Snarks are cyclically 4-edge-connected cubic graphs that do not allow a 3-edge-coloring.  Cavicchiolic et
    al.~\cite{101} proved that all the snarks of order less than 30 are of type-I (this was proved with the aid of a
    computer).  Also they proposed a open problem ``find (if any) the smallest snark (with respect to the order) which is of
    type-II". Motivated by that question, Campos et al.~\cite{71} proved that all graphs in three infinite families of
    snarks, the Flower Snarks, the Goldberg Snarks, and the Twisted Goldberg Snarks are type-I. They gave recursive
    procedures to construct total-colourings that uses 4 colours in each case. Also they proposed an open problem that all
    snarks are type-I. Sasaki el al.~\cite{102} prove that the total chromatic number of some classes of Snarks like,
    Loupekhine, Goldberg snarksboth and Blanuša’s families of graphs is 4. They observed that the total chromatic number
    seems to have no relation with the chromatic index for a cubic graph of cyclic-edge-connectivity less than 4. Also they
    proposed the following questions: \\ (i). What is the smallest type-II cubic graph without a square? \\ (ii). What is
    the smallest type-II snark?\\ Gunnar Brinkmann et al.~\cite{75} considered the problems posed by Campos et al.~\cite{71}
    and Sasaki el al.~\cite{102} and showed that there exists type-II snarks for each even order $n\geq 40.$ They gave
    a computer search for which all the cubic graphs with girth 5 and up to 32 vertices are type-I. Also they proposed the
    following questions:\\ (i). Does there exist a type-II snark of order less than 40? (The only possible orders for which
    the existence is not yet known are 36 and 38.)\\ (ii). What is the smallest type-II cubic graph with girth at least 5?\\
    (iii). Is there a girth $g$ so that all cubic graphs with girth at least $g$ are type-I?\\


    \subsection{Graphs with Degree Constraints} 
    Being a very difficult conjecture,     it makes sense to prove the conjecture either for known classes of graphs or graphs with some degree constrains.
    Hilton and Hind~\cite{25} showed that TCC holds for the graphs $G$ having $\Delta(G)\geq \frac{3}{4}|V(G)|$. Chetwynd et
    al.~\cite{43} gave a necessary and sufficient condition for $\chi''(G)=\Delta(G)+1$,  if $G$ is odd order and regular of
    degree $d\geq \frac{1}{3}\sqrt7 |V(G)|$. \\

    Deficiency of a graph $G$ to be def$(G)= \sum_{v\in V(G)} (\Delta(G) - d(v))$.   A graph $G$ is said to be conformable if
    $G$ has a vertex colouring that uses $\Delta(G) + 1$ colours with def$(G) > n$,  where $n$ is the number of colour
    classes with parity different from $|V(G)|$.  Chew~\cite{30} improved the previous result (Chetwynd et al.~\cite{43})
    for $d\geq [\frac{(\sqrt{37}-1)}{6}]|V(G)|$. He proved that for  any  regular graph $G$ of odd order and with $d\geq
    [\frac{(\sqrt{37}-1)}{6}]|V(G)|$,  $G$ is type-I if and only if $G$ is conformable; otherwise type-II.\\

    Dezheng Xie and Zhongshi He~\cite{26} showed that if $G$ is a regular graph of even order and $\delta(G)\geq
    \frac{2}{3}|V(G)|+ \frac{23}{6}$, then $\chi''(G)\leq \Delta(G)+2$.  Later, Xie DeZheng and Yang WanNian~\cite{55}
    proved the same result for regular graph of odd order. Combining these two results, we conclude that if $G$ regular
    graph with $\delta(G)\geq \frac{2}{3}|V(G)|+ \frac{23}{6}$ then $G$ satisfies TCC.  In~\cite{32} Machado and de
    Figueiredo proved that every non-complete \{square, unichord\}-free graph of maximum degree at least 4 is type-I.  Also
    they proved that any \{square, unichord\}-free graph is total colorable. Using graph decompositions, the same
    authors~\cite{33} proved that the non-complete \{ square, unichord \}-free graphs of maximum degree 3 are type-I.

    A graph is said to be $s$-degenerate for an integer $s\geq 1$ if it can be reduced to a trivial graph by successive
    removal of vertices with degree $ \leq s$. For example, every planar graph is 5-degenerate. Shuji Isobe et
    al.~\cite{37} proved that an $s$-degenerate graph $G$ has admits a total coloring with $\Delta(G)+1$ colors if the maximum degree
    $\Delta(G)\geq 4s+3$. The proof is based on Vizing’s and and Konig’s theorems on edge colorings. Further, they gave a linear
    time algorithm to find a total coloring of a graph $G$ with minimum number of colors if $G$ is a partial $k$-tree.


    \subsection{Other Classes of  Graphs}

    Mycielski ~\cite{42}, introduced the graph Mycielskian graph $\mu(G)$, to build a graph with a high chromatic number and
    a small clique number.  Let $G$ be a graph with vertex set $V^0=\{v_1^0,v_2^0,...,v_n^0\}$ and edge set $E^0$. Given an
    integer $m\geq 1$, the $m$-Mycielskian of $G$, denoted by $\mu_m(G)$,  is the graph with vertex set $V^0\cup V^1\cup
    ...\cup V^m \cup\{u\}$, where $V^i=\{v_j^i:v_j^0\in V^0\}$ is the $i^{th}$ distinct copy of $V^0$ for $i=1,2,...,m,$ and
    the edge set $E^0 \cup(\bigcup_{i=0}^{m-1}\{v_j^iv_j^{i+1}:v_j^0v_j'^0 \in E^0\})\cup \{v_j^m u:v_j^m \in V^m\}$. Chen
    et al.~\cite{29} showed that the generalized Mycielski graphs satisfy TCC. Also they proved  the total chromatic
    number of generalized Mycielski graphs $\mu_m(G)$ is $\Delta(\mu_m(G))+1$ if $\Delta(G)\leq \frac{|V(G)|-1}{2}$.

    Zhi-wen Wang et al.~\cite{38} proved that the vertex distinguishing total chromatic number and the total chromatic
    number  are same for the graphs $P_n\vee P_n$ and $C_n\vee C_n$. Li and Zhang~\cite{85} proved that the join of a
    complete inequibipartite graph $K_{n_1,n_2}$ and a path $P_m$ is type-I.  Hilton et al.~\cite{84}, determined the total
    chromatic numbers of graphs of the form $G_1+G_2$, where $G_1$ and $G_2$ are graphs of maximum degree at most two.\\
    The line graph of $G$, denoted by $L(G)$, has the set $E(G)$ as its
    vertex set and two distinct vertices $e_1, e_2 \in V (L(G))$ are
    adjacent if and only if they share a common vertex in $G$. Vignesh et al. ~\cite{VGS18} showed in a direct manner that for $n\leq 4$, $L(K_n)$ is type-I. They believe that $L(K_n)$ is always type-I. Hence they proposed the following conjecture:
    \begin{con}
        For any complete graph $K_n$, $\chi''(L(K_n))= 2n-3$.
    \end{con}
     The \textit{double graph} $D(G)$ of a given graph $G$ is constructed by making two copies of $G$ (including the initial edge set of each) and adding edges $((u,1),(v,2))$ and $((v,1),(u,2))$ for every edge $uv$ of $G$. Vignesh et al. ~\cite{VGS18} also prove that for any total colorable graph $G$,\\
    $\chi''(D(G)) \begin{cases}=\Delta(D(G))+1 & \text{if $G$ is type I}\\ \leq  \Delta(D(G))+2 & \text{if $G$ is type II}. \end{cases} $

    We know that middle graphs are subclasses of total graphs and super classes of line graphs. Muthuramakrishnan and Jayaraman ~\cite{MJ17}  obtained the total chromatic number for line, middle and total graphs of star and bistars.  \\
		
        The Kneser graph $K(n,k)$ is the graph whose vertices correspond to the $k$-element subsets of a set of $n$ elements,
        and where two vertices are adjacent if and only if the two corresponding sets are disjoint. A vertex in the odd graph
        $O_n$  is a $(n - 1)$-element subset of a $(2n -1)$-element set. Two vertices are connected by an edge if and only if
        the corresponding subsets are disjoint. Note that the odd graphs are particular case of Kneser graphs.

        \begin{thm} Odd graph $O_n$ satisfies TCC.
        \end{thm} \begin{proof}

            Consider a $2n-1$ element set $X$. Let $O_n$ be an odd graph defined from the subsets of $X$. Let $x$ be any element of
            $X$. Then, among the vertices of $O_n$, exactly $ \binom {2n-2} {n-2}$ vertices correspond to sets that contain $x$.
            Because all these sets contain $x$, they are not disjoint, and form an independent set of $O_n$. That is, $O_n$ has
            $2n - 1$ different independent sets of size $ \binom {2n-2} {n-2}$.  Further, every maximum independent set must
            have this form, so, $O_n$ has exactly $2n - 1$ maximum independent sets.

            If $I$ is a maximum independent set, formed by the sets that contain $x$, then the complement of $I$ is the set of
            vertices that do not contain $x$. This complementary set induces a matching in $G$. Each vertex of the independent
            set is adjacent to $n$ vertices of the matching, and each vertex of the matching is adjacent to $n - 1$ vertices of
            the independent set ~\cite{god80}.

            Based on the decomposition, we give a total coloring of $O_n$ in the following way:

            Assign $n$ colors to the edges between the vertices in the maximum independent set $I$ and the vertices in the matching.
            Color the edges in matching and the vertices in $I$  with a new color. Color one set of vertices in the matching
            with another new color and the second set of vertices with the missing colors at these vertices.  This will give a
            total coloring of $O_n$ using at most $n+2=\Delta(O_n)+2$ colors.

        \end{proof}

        In the next section we look at the algorithmic aspects of TCC that has been discussed in the literature.

        \section{Algorithms}

        It is known~\cite{81} that the problem of finding a minimal total coloring of a graph is in general case NP-hard.  In
        the same paper, Sanchez-Arroyo also proved that the problem is NP-complete even for cubic bipartite graphs. For general
        classes of graphs, the total coloring would be harder than edge colouring. Due to its complexity, several authors aim
        to find classes of graphs where there is a  polynomial time algorithm for optimal total coloring.  Bojarshinov~\cite{21}
        showed that the Behzad and Vizing conjecture holds for interval graphs. Also, he proved that every interval graph with
        even maximum degree can be totally colored in  $\Delta(G)+1$ colors in time $O(|V(G)|+|E(G)|+(\Delta(G))^2)$. This is
        the first known polynomial time algorithm for total colorings. Recently, Golumbic~\cite{53} showed that a
        rooted path graph $G$ is type-I if $\Delta(G)$ is even, otherwise it satisfies TCC. Also, he gave a  greedy algorithm
        (very greedy neighborhood coloring algorithm) which takes $O(|V(G)|+|E(G)|)$ time. \\

        Chordal graphs are a subclass of the perfect graphs. We know that linear time algorithms  exist for vertex colorings of
        chordal graphs. Yet, the complexity of total coloring is open for the class of chordal graphs. The complexity is
        know for interval
        graphs~\cite{21}, split graphs~\cite{27} and dually chordal graphs~\cite{41}. In~\cite{54}, Machado and  Celina de
        Figueiredo proved that the total coloring of bipartite unichord-free  graphs is NP-complete using the concept of separating
        class.  Machado et al. ~\cite{34} used a decomposition result to establish that every chordless graph of maximum degree
        $\Delta(G)\geq 3$ has total chromatic number $\Delta(G)+1$ and proved that this coloring can be obtained in time
        $O(|V(G)|^3|E(G)|)$. Machado et al. ~\cite{52}  discussed the time complexity of \{square, unichord\}-free graphs and
        showed that the total coloring can be obtained in polynomial time. In this case it is interesting to see that the edge coloring
        of this type of graphs is NP-complete. \\

        \begin{table}[h]
            \centering
            
            \begin{tabular} {| l | l| l |}
                \hline
                \textbf{Class of graphs }          & \textbf{Edge coloring}       & \textbf{Total coloring }           \\ \hline
                \hline
                Unichord free           & NP-complete~\cite{87}       & NP-complete~\cite{32}             \\ \hline

                Chordlesss              & Polynomial~\cite{34}        &       Polynomial~\cite{34}         \\ \hline

                \{Square, unichord \}-free, $\Delta \geq 4$ &   Polynomial~\cite{87}   &  Polynomial~\cite{33}        \\ \hline

                \{Square, unichord \}-free, $\Delta = 3$ &   NP-complete~\cite{87}   &  Polynomial~\cite{52}        \\ \hline

                Bipartite unichord free &   NP-complete~\cite{50}   &  NP-complete~\cite{54}       \\ \hline

                Interval graphs & Polynomial~\cite{21} & Polynomial~\cite{21} \\ \hline

                Some classes of circulant graphs & Polynomial~\cite{88} & Polynomial~\cite{18,39,61} \\ \hline

            \end{tabular}
            
            \caption{Computational complexity of edge and total colorings}

            \label{ccec}
        \end{table}
        Shuji Isobe et al. ~\cite{37} proved that the total coloring problem for $s$-degenerate graph can be solved in
        time O($n$log$n$) for a fairly large class of graphs including all planar graphs with sufficiently large maximum
        degree. Further, they showed that the total coloring problem can be solved in linear time for partial $k$-trees with bounded $k$.

        Dantas et al. ~\cite{73} proved that the problem of deciding whether the equitable total chromatic number is 4 is
        NP-complete for bipartite cubic graphs. They also found one family of type-I  cubic graphs of girth 5 having equitable
        total chromatic number 4. There are several classes of graphs in which the complexity of total coloring are unknown.
        We conclude this survey with a listing of the computational complexity of edge and total colorings of certain classes of graphs in Table~\ref{ccec}.


        \bibliographystyle{plain}
        \bibliography{rev}

\begin{thebibliography}{100}

\bibitem{100}
A.~Bahmanian and C.~A. Rodger.
\newblock What are graph amalgamations?
\newblock {\em eprint arXiv:1710.03844}, 2017.

\bibitem{BHO12}
J.~L. Baril, H.~Kheddouci, and O.~Togni.
\newblock Vertex distinguishing edge-and total-colorings of cartesian and other
  product graphs.
\newblock {\em Ars Comb.}, 107:109--127, 2012.

\bibitem{bsz18}
R.~Behr, V.~Sivaraman, and T.~Zaslavsky.
\newblock Mock threshold graphs.
\newblock {\em Discrete Mathematics}, 341(8):2159--2178, 2018.

\bibitem{76}
M.~Behzad.
\newblock {\em Graphs and and their chromatic numbers}.
\newblock PhD thesis, Michigan State University, 1965.

\bibitem{77}
M.~Behzad.
\newblock The total chromatic number.
\newblock In {\em Combinatorial Mathematics and Its Applications, Proc, Conf.,
  Oxford 1969}, pages 1--8, London, 1971.

\bibitem{bhjk08}
M.~Boggess, T.~J. Henderson, I.~Jimenez, and R.~Karpman.
\newblock The structure of unitary cayley graphs, technical report.
\newblock {\em semanticscholar.org}, 2008.

\bibitem{21}
V.~A. Bojarshinov.
\newblock Edge and total coloring of interval graphs.
\newblock {\em Discrete Applied Mathematics}, 114(1-3):23--28, 2001.

\bibitem{80}
O.~V. Borodin.
\newblock Colorings of plane graphs: A survey.
\newblock {\em Discrete Mathematics}, 313(4):517--539, 2013.

\bibitem{75}
G.~Brinkmann, M.~Preissmann, and D.~Sasaki.
\newblock Snarks with total chromatic number 5.
\newblock {\em Discrete Mathematics and Theoretical Computer Science},
  17(1):369--382, 2015.

\bibitem{68}
H.~Cai.
\newblock Total coloring of planar graphs without chordal 7-cycles.
\newblock {\em Acta Mathematica Sinica, English Series}, 31(12):1951--1962,
  2015.

\bibitem{70}
J.~S. Cai, G.~H. Wang, and G.~Y. Yan.
\newblock Planar graphs with maximum degree 8 and without intersecting chordal
  4-cycles are 9-totally colorable.
\newblock {\em Science China Mathematics}, 55(12):2601--2612, 2012.

\bibitem{71}
C.~N. Campos, S.~Dantas, and C.~P. de~Mello.
\newblock The total-chromatic number of some families of snarks.
\newblock {\em Discrete Mathematics}, 311(12):984--988, 2011.

\bibitem{28}
C.~N. Campos, C.~H. de~Figueiredo, R.~Machado, and C.~P. de~Mello.
\newblock The total chromatic number of split-indifference graphs.
\newblock {\em Discrete Mathematics}, 312(17):2690--2693, 2012.

\bibitem{39}
C.~N. Campos and C.~P. de~Mello.
\newblock Total colouring of $c_n^2$.
\newblock {\em Tendencias em Matematica Aplicada e Computacional},
  4(2):177--186, 2003.

\bibitem{31}
C.~N. Campos and C.~P. de~Mello.
\newblock The total chromatics number of some bipartite graphs.
\newblock {\em Electronic Notes in Discrete Mathematics}, 22:557--561, 2005.

\bibitem{18}
C.~N. Campos and C.~P. de~Mello.
\newblock A result on the total colouring of powers of cycles.
\newblock {\em Discrete Applied Mathematics}, 155(5):585--597, 2007.

\bibitem{101}
A.~Cavicchioli, T.~E. Murgolo, B.~Ruini, and F.~Spaggiari.
\newblock Special classes of snarks.
\newblock {\em Acta Applicandae Mathematicae}, 76(1):57--88, 2003.

\bibitem{63}
G.~J. Chang, J.~Hou, and N.~Roussel.
\newblock Local condition for planar graphs of maximum degree 7 to be 8-totally
  colorable.
\newblock {\em Discrete Applied Mathematics}, 159(8):760--768, 2011.

\bibitem{11}
J.~Chang, H.J. Wang, J.L. Wu, and A.~Yong-Ga.
\newblock Total colorings of planar graphs with maximum degree 8 and without
  5-cycles with two chords.
\newblock {\em Theoretical Computer Science}, 476:16--23, 2013.

\bibitem{17}
J.~Chang, J.L. Wu, and A.~Yong-Ga.
\newblock Total colorings of planar graphs with sparse triangles.
\newblock {\em Theoretical Computer Science}, 526:120--129, 2014.

\bibitem{89}
B.~L. Chen, L.~Dong, Q.~Z. Liu, and K.~C. Huang.
\newblock Total colorings of equibipartite graphs.
\newblock {\em Discrete mathematics}, 194(1-3):59--65, 1999.

\bibitem{chen92}
B.~L. Chen and H.~L. Fu.
\newblock Total colorings of graphs of order 2n having maximum degree 2n- 2.
\newblock {\em Graphs and Combinatorics}, 8(2):119--123, 1992.

\bibitem{27}
B.~L. Chen, H.~L. Fu, and M.~T. Ko.
\newblock Total chroamtic number and chromatic index of split graphs.
\newblock {\em Journal of Combinatorial Mathematics and Combinatorial
  Computing}, (17):137--146, 1995.

\bibitem{29}
M.~Chen, X.~Guo, H.~Li, and L.~Zhang.
\newblock Total chromatic number of generalized mycielski graphs.
\newblock {\em Discrete Mathematics}, 334:48--51, 2014.

\bibitem{43}
A.~G. Chetwynd, A.~J.~W. Hilton, and Z.~Cheng.
\newblock The total chromatic number of graphs of high minimum degree.
\newblock {\em Journal of the London Mathematical Society}, 2(s2-44):193--202,
  1991.

\bibitem{30}
K.~H. Chew.
\newblock Total chromatic number of regular graphs of odd order and high
  degree.
\newblock {\em Discrete Mathematics}, 154(1-3):41--51, 1996.

\bibitem{22}
T.~Chunling, L.~Xiaohui, Y.~Yuanshenga, and L.~Zhihe.
\newblock Equitable total coloring of $c_m$ and $c_n$.
\newblock {\em Discrete Applied Mathematics}, 157(4):596--601, 2009.

\bibitem{98}
D.~W. Cranston and D.~B. West.
\newblock An introduction to the discharging method via graph coloring.
\newblock {\em Discrete Mathematics}, 340(4):766--793, 2017.

\bibitem{2}
J.~Czap.
\newblock A note on total colorings of 1-planar graphs.
\newblock {\em Information Processing Letters}, 113(14-16):516--517, 2013.

\bibitem{56}
A.~Dalal, B.~S. Panda, and C.~A. Rodger.
\newblock Total-colorings of complete multipartite graphs using amalgamations.
\newblock {\em Discrete Mathematics}, 339(5):1587--1592, 2016.

\bibitem{82}
A.~Dalal and C.~A. Rodger.
\newblock The total chromatic number of complete multipartite graphs with low
  deficiency.
\newblock {\em Graphs and Combinatorics}, 31(6):2159--2173, 2015.

\bibitem{73}
S.~Dantas, C.~M.~H. de~Figueiredo, G.~Mazzuoccolo, M.~Preissmann, V.F. dos
  Santos, and D.~Sasaki.
\newblock On the equitable total chromatic number of cubic graphs.
\newblock {\em Discrete Applied Mathematics}, 209:84--91, 2016.

\bibitem{23}
S.~Dantas, C.M.H. de~Figueiredo, G.~Mazzuoccolo, M.~Preissmann, V.F. dos
  Santos, and D.~Sasaki.
\newblock On the total coloring of generalized petersen graphs.
\newblock {\em Discrete Mathematics}, 339(5):1471--1475, 2016.

\bibitem{41}
C.~M.~H. de~Figueiredo, J.~Meidanis, and C.~P. de~Mello.
\newblock Total-chromatic number and chromatic index of dually chordal graphs.
\newblock {\em Information Processing Letters}, 70(3):147--152, 1999.

\bibitem{48}
A.~Dong, G.~Liu, and G.~Li.
\newblock List edge and list total colorings of planar graphs without 6-cycles
  with chord.
\newblock {\em Bulletin of the Korean Mathematical Society}, 49(2):359--365,
  2012.

\bibitem{83}
L.~Dong and H.~P. Yap.
\newblock The total chromatic number of unbalanced complete $r$-partite graphs
  of even order.
\newblock {\em Bull. Inst. Comb. Appl}, 28:107--117, 2000.

\bibitem{96}
J.~Geetha, H.~L. Fu, and K.~Somasundaram.
\newblock Total colorings of circulant graphs.
\newblock {\em Submitted}, 2018.

\bibitem{57}
J.~Geetha and K.~Somasundaram.
\newblock Total coloring of generalized sierpi$\acute{\text n}$ski graphs.
\newblock {\em Australasian Journal Of Combinatorics}, 63(1):58--69, 2015.

\bibitem{58}
J.~Geetha and K.~Somasundaram.
\newblock Total colorings of product graphs.
\newblock {\em Graphs and Combinatorics}, 34(2):339--347, 2018.

\bibitem{god80}
C.~D. Godsil.
\newblock More odd graph theory.
\newblock {\em Discrete Mathematics}, 32(2):205--207, 1980.

\bibitem{53}
M.~C. Golumbic.
\newblock Total coloring of rooted path graphs.
\newblock {\em Information Processing Letters}, 135:73--76, 2018.

\bibitem{88}
J.~Hattingh.
\newblock The edge chromatic number of circulant.
\newblock {\em Quaestiones Mathematicae}, 11(4):371--381, 1988.

\bibitem{97}
A.~J.~W. Hilton.
\newblock A total-chromatic number analogue of plantholt's theorem.
\newblock {\em Discrete mathematics}, 79(2):169--175, 1990.

\bibitem{hil90}
A.~J.~W. Hilton.
\newblock A total-chromatic number analogue of plantholt's theorem.
\newblock {\em Discrete mathematics}, 79(2):169--175, 1990.

\bibitem{25}
A.~J.~W. Hilton and H.~R. Hind.
\newblock The total chromatic number of graphs having large maximum degree.
\newblock {\em Discrete Mathematics}, 117(1-3):127--140, 1993.

\bibitem{84}
A.~J.~W. Hilton, J.~Liu, and C.~Zhao.
\newblock The total chromatic numbers of joins of sparse graphs.
\newblock {\em Australasian Journal of Combinatorics}, 28:93--105, 2003.

\bibitem{20}
A.~M. Hinz and D.~Parisse.
\newblock Coloring hanoi and sierpi$\acute{\text n}$ski graphs.
\newblock {\em Discrete Mathematics}, 312(9):1521--1535, 2012.

\bibitem{8}
J.~Hou, B.~Liu, G.~Liu, and J.~Wu.
\newblock Total coloring of planar graphs without 6-cycles.
\newblock {\em Discrete Applied Mathematics}, 159(2-3):157--163, 2011.

\bibitem{103}
J.~Hou, Y.~Zhu, G.~Liu, J.~Wu, and M.~Lan.
\newblock Total colorings of planar graphs without small cycles.
\newblock {\em Graphs and Combinatorics}, 24(2):91--100, 2008.

\bibitem{37}
S.~Isobe, X.~Zhou, and T.~Nishizeki.
\newblock Total colorings of degenerate graphs.
\newblock {\em Combinatorica}, 27(2):167--182, 2007.

\bibitem{90}
M.~Jakovac and S.~Klav$\check{\text z}$ar.
\newblock Vertex-, edge-, and total-colorings of sierpi$\acute{\text
  n}$ski-like graphs.
\newblock {\em Discrete Mathematics}, 309(6):1548--1556, 2009.

\bibitem{36}
A.~Kemnitz and M.~Marangio.
\newblock Total colorings of cartesian products of graphs.
\newblock {\em Congressus Numerantium}, 165:99--110, 2003.

\bibitem{61}
R.~Khennoufa and O.~Togni.
\newblock Total and fractional total colourings of circulant graphs.
\newblock {\em Discrete Mathematics}, 308(24):6316--6329, 2008.

\bibitem{ws07}
W.~Klotz and T.~Sander.
\newblock Some properties of unitary cayley graphs.
\newblock {\em The electronic journal of combinatorics}, 14(1):45, 2007.

\bibitem{kos96}
A.~V. Kostochka.
\newblock The total chromatic number of any multigraph with maximum degree five
  is at most seven.
\newblock {\em Discrete Mathematics}, 162(1-3):199--214, 1996.

\bibitem{kss08}
L.~Kowalik, J.~S. Sereni, and R.~{\v{S}}krekovski.
\newblock Total-coloring of plane graphs with maximum degree nine.
\newblock {\em SIAM Journal on Discrete Mathematics}, 22(4):1462--1479, 2008.

\bibitem{85}
G.~Li and L.~Zhang.
\newblock Total chromatic number of one kind of join graphs.
\newblock {\em Discrete Mathematics}, 306(16):1895--1905, 2003.

\bibitem{50}
X.~Li.
\newblock Total coloring of planar graphs with maximum degree six.
\newblock {\em Recent Advances in Computer Science and Information Engineering,
  Lecture Notes in Electrical Engineering}, 129:801--805, 2012.

\bibitem{49}
B.~Liu, J.~F. Hou, and G.~Z. Liu.
\newblock List total colorings of planar graphs without triangles at small
  distance.
\newblock {\em Acta Mathematica Sinica English Series}, 27(12):2437--2444,
  2011.

\bibitem{54}
R.~Machado and C.~de~Figueiredo.
\newblock Complexity separating classes for edge-colouring and total-colouring.
\newblock {\em Journal of the Brazilian Computer Society}, 17(4):281--285,
  2011.

\bibitem{33}
R.~C.~S. Machado and C.~M.~H. de~Figueiredo.
\newblock Total chromatic number of \{square,unichord\}-free graphs.
\newblock {\em Electronic Notes in Discrete Mathematics}, 36:671--678, 2010.

\bibitem{32}
R.~C.~S. Machado and C.~M.~H. de~Figueiredo.
\newblock Total chromatic number of unichord-free graphs.
\newblock {\em Discrete Applied Mathematics}, 159(16):1851--1864, 2011.

\bibitem{87}
R.~C.~S. Machado, C.~M.~H. de~Figueiredo, and K.~Vuskovic.
\newblock Chromatic index of graphs with no cycle with a unique chord.
\newblock {\em Theoretical Computer Science}, 411(7-9):1221--1234, 2010.

\bibitem{34}
R.~C.~S. Machadoa, C.~M.~H. de~Figueiredo, and N.~Trotignon.
\newblock Edge-colouring and total-colouring chordless graphs.
\newblock {\em Discrete Mathematics}, 313(14):1547--1552, 2013.

\bibitem{52}
R.~C.~S. Machadoa, C.~M.~H. de~Figueiredo, and N.~Trotignon.
\newblock Complexity of colouring problems restricted to unichord-free and
  \{square,unichord\}-free graphs.
\newblock {\em Discrete Applied Mathematics}, 164:191--199, 2014.

\bibitem{59}
S.~Mohan, J.~Geetha, and K.~Somasundaram.
\newblock Total coloring of the corona product of two graphs.
\newblock {\em Australasian Journal Of Combinatorics}, 68(1):15--22, 2017.

\bibitem{86}
M.~Molloy and B.~Reed.
\newblock A bound on the total chromatic number.
\newblock {\em Combinatorica}, 18(2):241--280, 1998.

\bibitem{MJ17}
D.~Muthuramakrishnan and G.~Jayaraman.
\newblock Total chromatic number of star and bistargraphs.
\newblock {\em International Journal of Pure and Applied Mathematics},
  117(21):699--708, 2017.

\bibitem{42}
J.~Mycielski.
\newblock Sur le coloriage des graphs.
\newblock {\em Colloquium Mathematicum}, 3(9):161--162, 1955.

\bibitem{99}
H.~Petr.
\newblock Discharging technique in practice.
\newblock {\em Lecture text for Spring School on Combinatorics}, 2000.

\bibitem{24}
K.~Prnaver and B.~Zmazek.
\newblock On total chromatic number of direct product graphs.
\newblock {\em Journal of Applied Mathematics and Computing}, 33(1-2):449--457,
  2010.

\bibitem{62}
N.~Roussel.
\newblock Local condition for planar graphs of maximum degree 6 to be total
  8-colorable.
\newblock {\em Taiwanese Journal Of Mathematics}, 15(1):87--99, 2011.

\bibitem{7}
N.~Roussel and X.~Zhu.
\newblock Total coloring of planar graphs of maximum degree eight.
\newblock {\em Information Processing Letters}, 110(8-9):321--324, 2010.

\bibitem{93}
G.~Sabidussi.
\newblock Graph multiplication.
\newblock {\em Mathematische Zeitschrift}, 72(1):446--457, 1960.

\bibitem{81}
A.~Sanchez-Arroyo.
\newblock Determining the total colouring number is np-hard.
\newblock {\em Discrete Mathematics}, 78(3):315–319, 1989.

\bibitem{95}
D.~P. Sanders and Y.~Zhao.
\newblock On total 9-coloring planar graphs of maximum degree seven.
\newblock {\em Journal of Graph Theory}, 31(1):67--73, 1999.

\bibitem{102}
D.~Sasaki, S.~Dantas, C.~M.~H. de~Figueiredo, and M.~Preissmann.
\newblock The hunting of a snark with total chromatic number 5.
\newblock {\em Discrete Applied Mathematics}, 164:470--481, 2014.

\bibitem{40}
R.~S. Scorer, P.~M. Grundy, and C.~A.~B. Smith.
\newblock Some binary games.
\newblock {\em Math. Gaz.}, 28:96--103, 1944.

\bibitem{35}
M.~A. Seoud.
\newblock Total chromatic numebers.
\newblock {\em Applied Mathematics Letters}, 5(6):37--39, 1992.

\bibitem{SMWW97}
M.~A. Seoud, AEI~Abd el~Maqsoud, R.~J. Wilson, and J.~Williams.
\newblock Total colourings of cartesian products.
\newblock {\em International Journal of Mathematical Education in Science and
  Technology}, 28(4):481--487, 1997.

\bibitem{91}
L.~Shen and Y.~Wang.
\newblock On the 7 total colorability of planar graphs with maximum degree 6
  and without 4-cycles.
\newblock {\em Graphs and Combinatorics}, 25(3):401--407, 2009.

\bibitem{5}
L.~Shen and Y.~Wang.
\newblock Planar graphs with maximum degree 7 and without 5-cycles are
  8-totally-colorable.
\newblock {\em Discrete Mathematics}, 310(17-18):2372--2379, 2010.

\bibitem{60}
X.~Y. Sun, J.~L. Wu, Y.~W. Wu, and J.~F. Hou.
\newblock Total colorings of planar graphs without adjacent triangles.
\newblock {\em Discrete Mathematics}, 309(1):202--206, 2009.

\bibitem{VGS18}
R.~Vignesh, J.~Geetha, and K.~Somasundaram.
\newblock Total coloring conjecture for certain classes of graphs.
\newblock {\em Algorithms}, 11(10):161, 2018.

\bibitem{94}
V.~G. Vizing.
\newblock The cartesian product of graphs.
\newblock {\em Vyc.Sis,}, 9(1):30--43, 1963.

\bibitem{78}
V.~G. Vizing.
\newblock On an estimate of the chromatic class of a $p$-graph.
\newblock {\em Diskret Analiz}, 3:25--30, 1964.

\bibitem{46}
B.~Wang and J.~L. Wu.
\newblock Total coloring of planar graphs with maximum degree 7.
\newblock {\em Information Processing Letters}, 111(20):1019--1021, 2011.

\bibitem{13}
B.~Wang and J.~L. Wu.
\newblock Total colorings of planar graphs with maximum degree seven and
  without intersecting 3-cycles.
\newblock {\em Discrete Mathematics}, 311(18-19):2025--2030, 2011.

\bibitem{14}
B.~Wang and J.~L. Wu.
\newblock Total colorings of planar graphs without intersecting 5-cycles.
\newblock {\em Discrete Applied Mathematics}, 160(12):1815--1821, 2012.

\bibitem{12}
B.~Wang, J.~L. Wu, and S.~F. Tian.
\newblock Total colorings of planar graphs with small maximum degree.
\newblock {\em Bulletin of the Malaysian Mathematical Sciences Society},
  36(3):783--787, 2013.

\bibitem{15}
B.~Wang, J.~L. Wu, and H.~J. Wang.
\newblock Total colorings of planar graphs without chordal 6-cycles.
\newblock {\em Discrete Applied Mathematics}, 171:116--121, 2014.

\bibitem{9}
H.~Wang, B.~Liu, Y.~Gu, X.~Zhang, W.~Wu, and H.~Gao.
\newblock Total coloring of planar graphs without adjacent short cycles.
\newblock {\em Journal of Combinatorial Optimization}, 33(1):265--274, 2017.

\bibitem{10}
H.~Wang, B.~Liu, and J.~Wu.
\newblock Total coloring of planar graphs without chordal short cycles.
\newblock {\em Graphs and Combinatorics}, 31(5):1755--1764, 2015.

\bibitem{16}
H.~Wang, L.~Wu, and J.~Wu.
\newblock Total coloring of planar graphs with maximum degree 8.
\newblock {\em Theoretical Computer Science}, 522:54--61, 2014.

\bibitem{47}
H.~Wang, L.~Wu, W.~Wu, P.~M. Pardalos, and J.~Wu.
\newblock Minimum total coloring of planar graph.
\newblock {\em Journal of Global Optimization}, 60(4):777--791, 2014.

\bibitem{69}
H.~J. Wang, Z.~Y. Luo, Y.~Gu, and H.~W. Gao.
\newblock A note on the minimum total coloring of planar graphs.
\newblock {\em Acta Mathematica Sinica, English Series}, 32(8):967--974, 2016.

\bibitem{1}
H.~J. Wang and J.~L. Wu.
\newblock A note on the total coloring of planar graphs without adjacent
  4-cycles.
\newblock {\em Discrete Matematics}, 312(11):1923--1926, 2012.

\bibitem{72}
T.~Wang.
\newblock Total coloring of 1-toroidal graphs with maximum degree at least 11
  and no adjacent triangles.
\newblock {\em Journal of Combinatorial Optimization}, 33(3):1090--1105, 2017.

\bibitem{92}
Y.~Wang, Q.~Li, and M.~Shangguan.
\newblock On total chromatic number of planar graphs without 4-cycles.
\newblock {\em Science in China Series A: Mathematics}, 50(1):81--86, 2007.

\bibitem{3}
Y.~Wang and W.~Wang.
\newblock Vertex distinguishing total colorings of outerplanar graphs.
\newblock {\em Journal of Combinatorial Optimization}, 19(2):123--133, 2010.

\bibitem{38}
Z.~Wang, L.~Yan, and Z.~Zhang.
\newblock Vertex distinguishing equitable total chromatic number of join graph.
\newblock {\em Acta Mathematicae Applicatae Sinica}, 23(3):433--438, 2007.

\bibitem{6}
Q.~Wu, Q.~Lu, and Y.~Wang.
\newblock $\delta+1$-total colorability of plane graphs of maximum degree
  $\delta \geq 6$ with neither chordal 5-cycle nor chordal 6-cycle.
\newblock {\em Information Processing Letters}, 111(15):767--772, 2011.

\bibitem{26}
D.~Xie and Z.~He.
\newblock The total chromatic number of regular graphs of even order and high
  degree.
\newblock {\em Discrete Mathematics}, 300(1-3):196--212, 2005.

\bibitem{55}
D.~Z. Xie and W.~N. Yang.
\newblock The total chromatic number of regular graphs of high degree.
\newblock {\em Science in China Series A: Mathematics}, 52(8):1743--1759, 2009.

\bibitem{74}
R.~Xu, J.~Wu, and H.~Wang.
\newblock Total coloring of planar graphs without some chordal 6-cycles.
\newblock {\em Bulletin of the Malaysian Mathematical Sciences Society},
  38(2):561--569, 2015.

\bibitem{66}
R.~Xu and J.~L. Wu.
\newblock Total coloring of planar graphs with 7-cycles containing at most two
  chords.
\newblock {\em Theoretical Computer Science}, 520:124--129, 2014.

\bibitem{79}
H.~P. Yap.
\newblock Total colourings of graphs.
\newblock {\em Lecture Notes in Mathematics, Springer}, 1623, 1996.

\bibitem{65}
J.~Zhang and Y.~Wang.
\newblock $(\delta+1)$-total-colorability of plane graphs with maximum degree
  $\delta$ at least 6 and without adjacent short cycles.
\newblock {\em Information Processing Letters}, 110(18-19):830--834, 2010.

\bibitem{64}
X.~Zhang.
\newblock List total coloring of pseudo-outerplanar graphs.
\newblock {\em Discrete Mathematics}, 313(20):2297--2306, 2013.

\bibitem{4}
X.~Zhang, J.~Hou, and G.~Liu.
\newblock On total colorings of 1-planar graphs.
\newblock {\em Journal of Combinatorial Optimization}, 30(1):160--173, 2015.

\bibitem{67}
X.~Zhang and G.~Liu.
\newblock Total coloring of pseudo-outerplanar graphs.
\newblock {\em arXiv}, 2011.

\bibitem{45}
X.~Zhang, J.~Wu, and G.~Liu.
\newblock List edge and list total coloring of 1-planar graphs.
\newblock {\em Frontiers of Mathematics in China}, 7(5):1005--1018, 2012.

\bibitem{51}
E.~Zhu and J.~Xu.
\newblock A sufficient condition for planar graphs with maximum degree 6 to be
  totally 8-colorable.
\newblock {\em Discrete Applied Mathematics}, 223:148--153, 2017.

\bibitem{19}
B.~Zmazek and J.~Zerovnik.
\newblock Behzad-vizing conjecture and cartesian-product graphs.
\newblock {\em Electronic Notes in Discrete Mathematics}, 15(17):297--300,
  2004.

\end{thebibliography}

    \end{document}